\documentclass[a4paper,12pt]{article}

\usepackage{geometry}
\usepackage{indentfirst}

\usepackage{amsmath, amssymb, amsfonts}
\usepackage{mathrsfs}
\usepackage{bm}
\usepackage{bbm}

\usepackage{amsthm}
\theoremstyle{definition}

\newtheorem{The}{Theorem}
\newtheorem{Lem}[The]{Lemma}
\newtheorem{Cor}[The]{Corollary}
\newtheorem{Pro}[The]{Proposition}

\newtheorem{Cla}{Claim}

\newtheorem{Not}[The]{Note}

\usepackage{booktabs}
\usepackage{longtable}
\usepackage{multirow}
\usepackage{tabularx}
\usepackage{threeparttable}
\usepackage{float}

\usepackage{graphicx}
\usepackage{rotating}
\usepackage{subcaption}
\usepackage{caption}
\usepackage{tikz}

\usepackage[noend]{algpseudocode}
\usepackage{algorithm}

\usepackage[numbers,sort&compress]{natbib}

\usepackage{enumerate}
\usepackage{multicol}
\usepackage{pifont}
\numberwithin{equation}{section}
\usepackage[colorlinks,linkcolor=blue,citecolor=blue,urlcolor=blue]{hyperref}
\usepackage[normalem]{ulem}

\usepackage{graphicx}

\begin{document}
		\title{ Bricks in which  every vertex is incident with a forcing edge
		}
	\author{
		Xinyu Dai$^{1}$\quad
		Fuliang Lu$^{1,2}$\quad
		Yaxian Zhang$^{3}$
	\thanks{Corresponding author.}
	}

	\date{
		{\small
			$^{1}$School of Mathematics and Statistics, Minnan Normal University,
			Zhangzhou, Fujian 363000, P.R. China\\
			$^{2}$Fujian Key Laboratory of Granular Computing and Applications,
			Minnan Normal University, Zhangzhou, Fujian 363000, P.R. China\\
			$^{3}$School of Mathematical Sciences, Sichuan Normal University,
			Chengdu, Sichuan 610068, P.R. China\\[0.5ex]
			E-mails: daixinyu03520@163.com, flianglu@163.com,
			yaxianzhang@sicnu.edu.cn
		}
	}

	\maketitle

	\begin{abstract}
A matching covered
graph is a brick if it is 3-connected and bicritical. An edge of a matching covered graph $G$ is a forcing edge if it
lies in precisely one perfect matching of $G$.
We prove that every vertex of a simple brick $G$ is incident with
a forcing edge if and only if $G$ is an odd wheel
or is isomorphic to one of the following three graphs:  the
triangular prism $\overline{C_6}$, $\overline{C_6}^{+}$ or the
bicorn $H_8$. Here $\overline{C_6}^{+}$ is obtained from
$\overline{C_6}$ by adding an edge between two nonadjacent
vertices.

	\end{abstract}


	\section{Introduction}

    All graphs considered here are finite and loopless; multiple edges are allowed unless otherwise stated. The \emph {underlying graph} of $G$, denoted by $\widetilde G$, is the graph obtained from $G$ by deleting all but one of the edges joining each pair of adjacent vertices.
	For the terminology that is specific to matching covered graphs, we follow Lov\'asz and Plummer~\cite{Lovasz1986}.
	A \emph{perfect matching} is a set of edges incident with every vertex exactly once. A graph is \emph{matchable} if it has a perfect matching.
	A connected graph with at least one edge is \emph{matching covered} if each edge belongs to some perfect matching.
	A graph $G$ is \emph{bicritical} if it contains at least four vertices, and $G-u-v$ is matchable for each pair of distinct vertices $u$ and $v$ of $G$.
		For a nonempty proper subset $X$ of $V(G)$, the edge set $\{xy\in E(G): x\in X,\ y\in V(G)\setminus X\}$
		is called an \emph{edge cut} of $G$ and denoted by $\partial_G(X)$.
		The vertex sets $X$ and $\overline{X}$ are called the \emph{shores} of the edge cut $\partial_G(X)$, where $\overline{X} = V(G) \setminus X$.
	The edge cut $\partial_G(X)$ is \emph{trivial} if one of its shores has exactly one vertex, and it is \emph{tight} in a matching covered graph if every perfect matching {shares exactly one edge with it}.
	A nonbipartite matching covered graph without nontrivial tight cuts is called a \emph{brick}.
    Equivalently, by the theorem of Edmonds, Lov\'asz and Pulleyblank~\cite{Lovasz1986},
    bricks are 3-connected bicritical graphs.
	Bricks are the nonbipartite atoms in the tight cut decomposition of matching covered graphs, a theory developed by Kotzig, Lov\'asz, Plummer and others; see the monographs of Lov\'asz and Plummer~\cite{Lovasz1986} and Lucchesi and Murty~\cite{Lucchesi2024}.
	A central result is Lov\'asz's theorem~\cite{Lovasz1987}, which states that any two tight cut decompositions of a matching covered graph yield the same list of bricks and braces, up to multiple edges. Let $b(G)$ denote the number of bricks in a tight cut decomposition
		of a matching covered graph $G$.

For a graph $G$ with a perfect matching, an edge $uv$ of $G$ is
called a \emph{forcing edge} (also known as \emph{solitary} in
\cite{Lucchesi2024}) if $G-u-v$ has a unique perfect matching;
equivalently, $uv$ is contained in exactly one perfect matching
of $G$.
The idea arose in chemical graph theory in connection with
Kekul\'e structures, first through the innate degree of freedom
of Klein and Randi\'c~\cite{Klein1987} and later through the
terminology of forcing used by Harary, Klein and
\v{Z}ivkovi\'c~\cite{Harary1991}; see also Adams, Mahdian and
Mahmoodian~\cite{Adams2004} and the survey~\cite{Zhang2025b}.
	Jiang and Zhang~\cite{Jiang2011} studied the forcing matching
	number of boron--nitrogen fullerene graphs. Wu, Ye and
	Zhang~\cite{Ye2016} showed that every $3$-connected cubic graph
	with a forcing edge can be obtained from $K_4$ by a sequence of
	$Y\to\Delta$-operations.
Recently, Goedgebeur et~al.~\cite{Goedgebeur2026} studied forcing
edges in bridgeless cubic graphs. They reduced the problem to the
corresponding class of $3$-connected cubic graphs and proved that
every graph in this class has at most six forcing edges.
	In related work, Narayana et~al.~\cite{Narayana2024} studied
	solitary edges in $3$-edge-connected $r$-graphs and, in
	particular, also obtained the upper bound of six for
	$3$-connected cubic graphs.
It is easy to check that a forcing edge $e$ of a matching covered
graph $G$ remains forcing in any brick $H$ obtained from a tight
cut decomposition of $G$, provided that $e\in E(H)$.
   With the help of forcing edges, de Carvalho, Lucchesi and Murty \cite{Carvalho2004} present a characterization of extremal bricks.
   Motivated by these results, we consider a vertex version of forcing edge conditions in bricks.
   More precisely, we study bricks with the following property $\mathcal{P}$: every vertex is incident with at least one forcing edge.

	A \emph{wheel} is a graph obtained from a cycle, called the \emph{rim}, by adding a new vertex, called the \emph{hub}, adjacent to every vertex of the rim.
	The vertices and edges of the rim are called \emph {rim vertices} and \emph{rim edges}, and the edges joining the hub to the rim vertices are called \emph{spokes}.
	A wheel is \emph{odd} if its rim has odd length.
	It can be checked that in an odd wheel every spoke is a forcing edge,
	and hence every odd wheel has the property $\mathcal{P}$.

We shall also use the following three graphs shown in
Figure~\ref{fig:T6}. The graph
$\overline{C_6}$ is the triangular prism obtained from two
vertex-disjoint triangles $xyzx$ and $x'y'z'x'$ by adding
the edges $xx'$, $yy'$ and $zz'$. The graph
$\overline{C_6}^{+}$ is obtained from
$\overline{C_6}$ by adding an edge between two nonadjacent
vertices.  Finally,
$H_8$, also known as the bicorn, is obtained from two
vertex-disjoint triangles $a_0a_1a_2a_0$ and $b_0b_1b_2b_0$
by adding two additional vertices $s_1$ and $s_2$ and the
edges $a_0b_0$, $a_1s_1$, $b_1s_1$, $a_2s_2$, $b_2s_2$ and
$s_1s_2$.
	\begin{figure}[H]
		\centering
		\includegraphics[width=\textwidth]{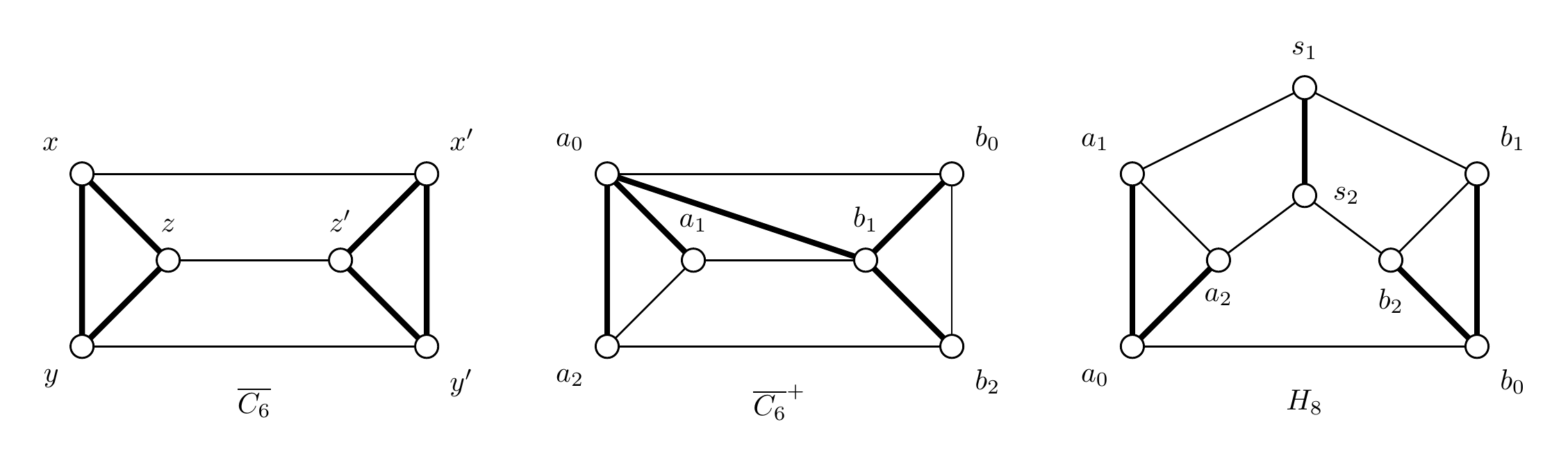}
		\caption{The graphs $\overline{C_6}$, $\overline{C_6}^{+}$ and $H_8$: the
			thick edges are forcing edges.}
		\label{fig:T6}
	\end{figure}

\begin{The}\label{thm:intro-main}
	Let $G$ be a simple brick. Then $G$ has the property $\mathcal P$
	if and only if $G$ is an odd wheel or is isomorphic to one of
	the graphs $\overline{C_6}$, $\overline{C_6}^{+}$ and $H_8$.
\end{The}

 The terminology concerning solid bricks, robust cuts and splicing
	is introduced in Section~\ref{sec:preliminaries}. In the same section, we prove that a
	simple solid brick with the property $\mathcal P$ is an odd wheel. In Section~\ref{sec:splicing}, we study
	splicings of an odd wheel up to multiple edges and a brick. 
	Finally, in Section~\ref{sec:main}, we prove Theorem~\ref{thm:intro-main} using
	a robust cut and induction on the number of robust cuts of a brick.


	\section{Preliminaries}\label{sec:preliminaries}

For a vertex $v\in V(G)$, let $\partial_G(v)$ be the set of all
edges incident with $v$. The \emph{degree} of $v$, denoted by
$d_G(v)$, is defined by $d_G(v)=|\partial_G(v)|$. We denote by
$N_G(v)$ the set of distinct neighbors of $v$. A vertex of degree $k$ is also called a \emph{$k$-degree} vertex.
	For $x,y\in V(G)$, a path with ends $x$ and $y$ is called an
	\emph{$x$-$y$ path}. For a graph $G$ with a matching $M$, a path is called an \emph{$M$-alternating path} if its edges alternate between $M$ and $E(G)\setminus M$.
For $X\subseteq V(G)$, let $G/(X\to x)$ denote the graph obtained
from $G$ by contracting $X$ to a single vertex $x$. When the label
$x$ is irrelevant, we simply write $G/X$. For an edge cut
$C=\partial_G(X)$, the graphs $G/X$ and $G/\overline X$ are called
the \emph{$C$-contractions} of $G$.

	\subsection{Cuts and Splicing}

	We shall use the following splicing operation.
	Let $G$ and $H$ be vertex-disjoint graphs, and let $u\in V(G)$ and $v\in V(H)$ satisfy $d_G(u)=d_H(v)$.
	Given a bijection $\theta:\partial_H(v)\to \partial_G(u)$, the graph obtained by splicing $G$ at $u$ with $H$ at $v$, denoted by
	$(G(u)\odot H(v))_\theta$, is obtained from $G-u$ and $H-v$ by joining, for each edge $vw\in \partial_H(v)$, the vertex $w$ to the other end of $\theta(vw)$ in $G$.
	The set of these new edges is called the \emph{splicing cut}.
	
  An edge cut $C$ of a matching covered graph $G$ is called a
  \emph{separating cut} if both $C$-contractions of $G$ are
  matching covered.
  A matching covered graph $G$ is a \emph{near-brick} if $b(G)=1$.
   A separating cut $C$ of a matching covered graph $G$ is
  \emph{robust} if $C$ is not tight and both $C$-contractions of
  $G$ are near-bricks.
   If $S\subseteq V(G)$ and $F\subseteq E(G)$, we say that $S$
   \emph{covers} $F$ if every edge of $F$ has an end in $S$, and
   that $F$ \emph{saturates} $S$ if every vertex of $S$ is incident
   with an edge of $F$.


	\begin{Lem}\label{H}
		\cite{Lovasz1986}
		If a bipartite graph $H$ has at least one edge and a unique
		perfect matching, then each part of a bipartition of $H$ contains
		a vertex $v$ with $d_H(v)=1$.
	\end{Lem}

\begin{Lem}\label{lem:bipartite}
	Assume that $\partial(X_i)$ is an edge cut of a brick $G$ such
	that $G_i=G/\overline{X_i}$ is a brick for each
	$i\in\{1,2\}$, and that
	$(G/(X_1\to x_1))/(X_2\to x_2)$ is a matching covered
	bipartite graph $H$. If $|V(H)|\geq4$, then every edge of $H$
	not joining $x_1$ and $x_2$ is nonforcing in $H$.
\end{Lem}

\begin{proof}
Let $A$ and $B$ be the two parts of a bipartition of $H$, and let $uv$ be
an edge of $H$ not joining $x_1$ and $x_2$, where $u\in A$ and
$v\in B$. Suppose, to the contrary, that $uv$ is forcing in $H$.
Then $H-u-v$ has a unique perfect matching. Since
$|V(H)|\geq4$, the graph $H-u-v$ has at least one edge. By
Lemma~\ref{H},
there exist $a\in A\setminus\{u\}$ and
$b\in B\setminus\{v\}$ such that
$d_{H-u-v}(a)=d_{H-u-v}(b)=1$.
Since $H$ is matching covered and $|V(H)|\geq4$, every vertex
of $H$ has at least two distinct neighbors. Since $a$ is not
adjacent to $u$, it has exactly two distinct neighbors in $H$,
one of which is $v$. Similarly, $b$ has exactly two distinct
neighbors in $H$, one of which is $u$.
Since $G$ is $3$-connected, each of $a$ and $b$ has at least
three neighbors in $G$. Thus each is joined by multiple edges
in $H$ to a contraction vertex. Since
$d_{H-u-v}(a)=d_{H-u-v}(b)=1$, these contraction vertices are
$v$ and $u$, respectively. Hence
$\{u,v\}=\{x_1,x_2\}$, contrary to the choice of $uv$.
Therefore, every edge of $H$ not joining $x_1$ and $x_2$ is
nonforcing in $H$.
\end{proof}

	\begin{Lem}\label{thm:splicing-brick}
		\cite{Carvalho2004}
		For $i=1,2$, let $G_i$ be a brick, let $x_i\in V(G_i)$, and let
		$G=(G_1(x_1)\odot G_2(x_2))_{\theta}$.
		Then $G$ is a brick if and only if no pair of vertices of $G$, one in
		$V(G_1)\setminus\{x_1\}$ and the other in $V(G_2)\setminus\{x_2\}$,
		covers the set of edges in $\partial_G(V(G_1)\setminus\{x_1\})$.
	\end{Lem}

	The following result follows directly from the correspondence between
	perfect matchings on the two shores of a splicing cut.

	\begin{Pro}\label{pro:M}
		Let $G_1$ and $G_2$ be matching covered graphs, let
		$G=(G_1(u)\odot G_2(v))_{\theta}$, and let $C$ be the splicing cut.
		For $ua\in\partial_{G_1}(u)$, write
		$\theta^{-1}(ua)=vb$, and let $ab$ be the corresponding edge of $C$.
		\begin{enumerate}[(i)]
			\item Let $e\in E(G_1-u)$. If $e$ is forcing in $G$, then $e$
			is forcing in $G_1$. Moreover, if the unique perfect matching of
			$G_1$ containing $e$ uses $ua$, then $vb$ is forcing in $G_2$,
			and every perfect matching of $G$ containing $e$
			{shares exactly one edge with $C$, namely $ab$}.
			\item If $ab$ is forcing in $G$, then $ua$ and $vb$ are forcing
			in $G_1$ and $G_2$, respectively, and every perfect matching of
			$G$ containing $ab$ {shares exactly one edge with $C$, namely $ab$}.
		\end{enumerate}
	\end{Pro}

	\begin{Pro}\label{pro:splicing-degree}
		Let $G\in G_1(v_1)\odot G_2(v_2)$, and let $C$ be the splicing cut.
		For $i\in\{1,2\}$, let $f\in E(G_i-v_i)$ be a forcing edge of
		$G_i$. If the unique perfect matching of $G_i$ containing $f$
		contains an edge joining $v_i$ and $w$, then
		$|\partial_G(w)\cap C|=1$.
	\end{Pro}

	\begin{proof}
		By the definition of splicing, the edge joining $v_i$ and $w$ corresponds to an edge of $C$
		incident with $w$. If $|\partial_G(w)\cap C|\geq2$, then $G_i$
		contains at least two distinct edges joining $v_i$ and $w$.
		Replacing one by another in the unique perfect matching containing
		$f$ gives a second perfect matching containing $f$, a contradiction.
	\end{proof}

	\begin{Pro}\label{pro:underlying-simple}
		Let $G$ be a graph, possibly with multiple edges. Then $G$ is a
		brick if and only if $\widetilde G$ is a brick. Moreover, if
		$G$ has the property $\mathcal P$, then $\widetilde G$ also has the
		property $\mathcal P$.
	\end{Pro}

	\subsection{Solid Bricks and Forcing Edges}

		 A brick is \emph{solid} if its separating cuts are trivial.
	A graph $G$ has \emph {odd cycle property} if for any two vertex-disjoint odd cycles $C_1$ and $C_2$ of $G$, the graph
	$G-V(C_1)-V(C_2)$ has no perfect matching.

\begin{The}\label{lem:robustcut}
	\cite{Carvalho2004,He2026}
	Every nonsolid brick $G$ has a robust cut $\partial(X)$ such that there exists a subset $X'$ of $X$ and a subset $X''$ of $\overline{X}$ such that
	$G / \overline{X'}$ is a solid brick,
	$G / \overline{X''}$ is a brick,
	and the graph $H$, obtained from $G$ by contracting $X'$ and $X''$ to single vertices
	$x'$ and $x''$, respectively, is bipartite and matching covered, where $x'$ and $x''$ lie in different parts of a bipartition of $H$.
\end{The}

		\begin{Lem}\label{lem:solid oddcycle}
				\cite{CLM2004,Lucchesi2024}
				For a brick $G$, it is solid if and only if it has odd cycle property.
			\end{Lem}
		The odd cycle property is inherited by subgraphs whose complements have perfect matchings.

	\begin{Lem}\label{lem:inherited oddcycle}
		Let $G$ be a graph with odd cycle property.
		If $H$ is a subgraph of $G$ such that $G-V(H)$ has a perfect matching, then $H$ has odd cycle property.
	\end{Lem}
	\begin{proof}
		Suppose that $H$ contains two vertex-disjoint odd cycles $C_1$ and $C_2$ such that
		$H-V(C_1)-V(C_2)$ has a perfect matching $M_H$.
		Let $M_0$ be a perfect matching of $G-V(H)$.
		Then $M_H\cup M_0$ is a perfect matching of $G-V(C_1)-V(C_2)$, contradicting the hypothesis on $G$.
	\end{proof}

       The next two results mainly concern the existence of 1-degree vertices; this fact will be crucial in the proof of the main structural conclusion.
	\begin{Lem}\label{lem:unique pm}
		\cite{Zhang2022}
		Let $G$ be a graph with a unique perfect matching.
		If $G$ has odd cycle property, then $G$ contains a 1-degree vertex.
	\end{Lem}

	\begin{Lem}\label{lem:forcing edge}
		Let $uv$ be a forcing edge of a solid brick $G$. Then $G-u-v$ has at least two 1-degree vertices, which are adjacent to both $u$ and $v$.
	\end{Lem}
	\begin{proof}
		Set $H=G-u-v$.
		Since $uv$ is forcing, $H$ has a unique perfect matching $M$.
		By Lemma~\ref{lem:solid oddcycle}, $G$ has odd cycle property.
		By Lemma~\ref{lem:inherited oddcycle}, $H$ satisfies the hypothesis of Lemma~\ref{lem:unique pm}; hence $H$ has a 1-degree vertex $x$.
		As $G$ is a brick, it is 3-connected. So $x$ is adjacent to both $u$ and $v$ in $G$.
		Suppose that $x$ is the only 1-degree vertex of $H$.
		Let $P$ be a maximal $M$-alternating path in $H$ starting with the unique edge of $M$ incident with $x$.
		Then $P$ ends with an edge of $M$; let $y$ be the end of $P$, other than $x$.
		Since $y$ is not a 1-degree vertex in $H$, the maximality of $P$ implies that the  end of each edge in $\partial(y)\setminus M $ other than $y$, lies on  $P$. Note that each edge in $\partial(y)\setminus M $ together with edges in $P$
		  cannot form any even $M$-alternating cycles, because $M$ is the unique perfect matching of $H$; let $e\in\partial(y)\setminus M$ such that $e$ forms some odd cycle $C_2$ with a subpath of $P$ and
		the triangle $C_1=xuvx$ is disjoint from $C_2$. Therefore,  $G-V(C_1)-V(C_2)$ has a perfect matching $(M \setminus E(P))\cup ((E(P)\setminus E(C_2))\setminus M)$, contradicting Lemma~\ref{lem:solid oddcycle}.
		Therefore $H$ has a 1-degree vertex other than $x$, say $z$. Again, by the 3-connectedness of $G$, $z$ is adjacent to both $u$ and $v$.
		\end{proof}

	The following observation follows directly from Lemma~\ref{lem:forcing edge}.
	\begin{Cor}\label{cor:observation}
	Let $G$ be a solid brick, and let $ux$ be a forcing edge with
	$d_G(x)=3$ and $N_G(x)=\{u,x^-,x^+\}$. Then
	$d_G(x^-)=d_G(x^+)=3$,
	$\{ux^-,ux^+\}\subseteq E(G)$, and
	the set of all $1$-degree vertices of $G-u-x$ is
		$\{x^-,x^+\}$.
	\end{Cor}

	\subsection{Odd Wheels}

	We first study how forcing edges restrict the structure of solid bricks. Then we study how splicing affects the property $\mathcal P$. 
		\begin{Lem}\label{lem:oddwheel}
			Let $G$ be a simple solid brick. If at most one vertex of $G$ is not incident with a forcing edge, then $G$ is an odd wheel.
		\end{Lem}

	\begin{proof}

	If there exists a vertex that is not incident with any forcing edge, denote the vertex by $z$.
	    Choose a forcing edge $e=uv$ of $G$, where $\{u,v\}\subseteq V(G)\setminus \{z\}$. By Lemma~\ref{lem:forcing edge}, $G-u-v$ has two $1$-degree vertices, say $x_1$ and $y_1$, and each is adjacent to both $u$ and $v$. Thus $d_G(x_1)=d_G(y_1)=3$.

	At least one of $x_1$ and $y_1$ is not the exceptional vertex $z$; without loss of generality, assume that $x_1\neq z$.
	Let $h_1$ be a forcing edge incident with $x_1$, and let $x_2\in N_G(x_1)\setminus \{u,v\}$
	and $y_2\in N_G(y_1)\setminus\{u,v\}$. If $x_2 = y_1$, then $x_1=y_2$. Hence the edge $x_1y_1$ forms a component of $G-u-v$. Since $G$ is 3-connected, $G-u-v$ is connected.
      If there is a vertex other than $x_1$ and $y_1$ in $G-u-v$, then it would be adjacent to $x_1$ or $y_1$, contradicting that $N_G(x_1)=\{u,v,y_1\}$ and $N_G(y_1)=\{u,v,x_1\}$. Therefore $V(G)=\{u,v,x_1,y_1\}$ and $|V(G)|=4$.
	Since $G$ is a brick, $G\cong K_4$.

	 Next assume that $x_2 \neq y_1$. If $h_1=x_1x_2$, then, applying Corollary~\ref{cor:observation} to the forcing edge $x_1x_2$, we obtain $d_G(u)=d_G(v)=3$ and $\{x_2u,x_2v\}\subseteq E(G)$.
	However, $x_2,x_1,v$ and $y_1$ are four distinct neighbors of $u$, contradicting $d_G(u)=3$.
	So $h_1=ux_1$ or $h_1=vx_1$. By the symmetry of $u$ and $v$, assume that $h_1=ux_1$.

	Applying Corollary~\ref{cor:observation} to the forcing edge $ux_1$, the vertices $v$ and $x_2$ have degree three and are adjacent to $u$. We now extend the path from $x_1$ through $x_2$ as follows. Suppose that $x_i$ has already been chosen, where $i\geq 2$, and let
     $N_G(x_i)=\{u,x_{i-1},x_{i+1}\}$.
	If $x_i\neq z$, then $x_i$ is incident with a forcing edge. This forcing edge cannot be $x_{i-1}x_i$ or $x_ix_{i+1}$, for otherwise Corollary~\ref{cor:observation} would imply $d_G(u)=3$, while $u$ is adjacent to the four distinct vertices $v,x_1,y_1$ and $x_2$, a contradiction. Hence $ux_i$ is a forcing edge. Applying Corollary~\ref{cor:observation} again, we obtain that $x_{i+1}$ has degree three and is adjacent to $u$.

	Continuing this process, we obtain a maximal sequence
	$
		x_1,x_2,\ldots,x_t
	$
	such that $ux_i$ is a forcing edge and $d_G(x_i)=3$ for $1\leq i<t$, $x_ix_{i+1}\in E(G)$ for $1\leq i<t$, and each $x_i$ is adjacent to $u$. If $x_t\neq z$, then $ux_t$ is also a forcing edge and $d_G(x_t)=3$. If $x_t=z$, then $d_G(z)=3$ and $uz\in E(G)$ by Corollary~\ref{cor:observation} applied to the forcing edge $ux_{t-1}$. By maximality, the sequence stops only when the next vertex already belongs to $\{v,x_1,\ldots,x_t\}$, or when the exceptional vertex $z$ exists and $x_t=z$.

	If the sequence closes before reaching $z$ (or if no such vertex $z$ exists), then the closed sequence gives a cycle; in particular, when it stops at $v$ (i.e., $x_t=y_1$), the cycle is
	$
		vx_1x_2\cdots x_t v
	$
	and, together with the vertex $u$, forms a wheel. Since every rim vertex has degree three, every neighbor of a rim vertex lies in the wheel. As $G$ is 3-connected, this wheel is the whole graph. Thus $G$ is a wheel. Since a brick has an even number of vertices, the rim has odd length, and so $G$ is an odd wheel, see Figure~\ref{fig:lemma12}(a).

	\begin{figure}[htbp]
		\centering
		\includegraphics[width=0.85\textwidth]{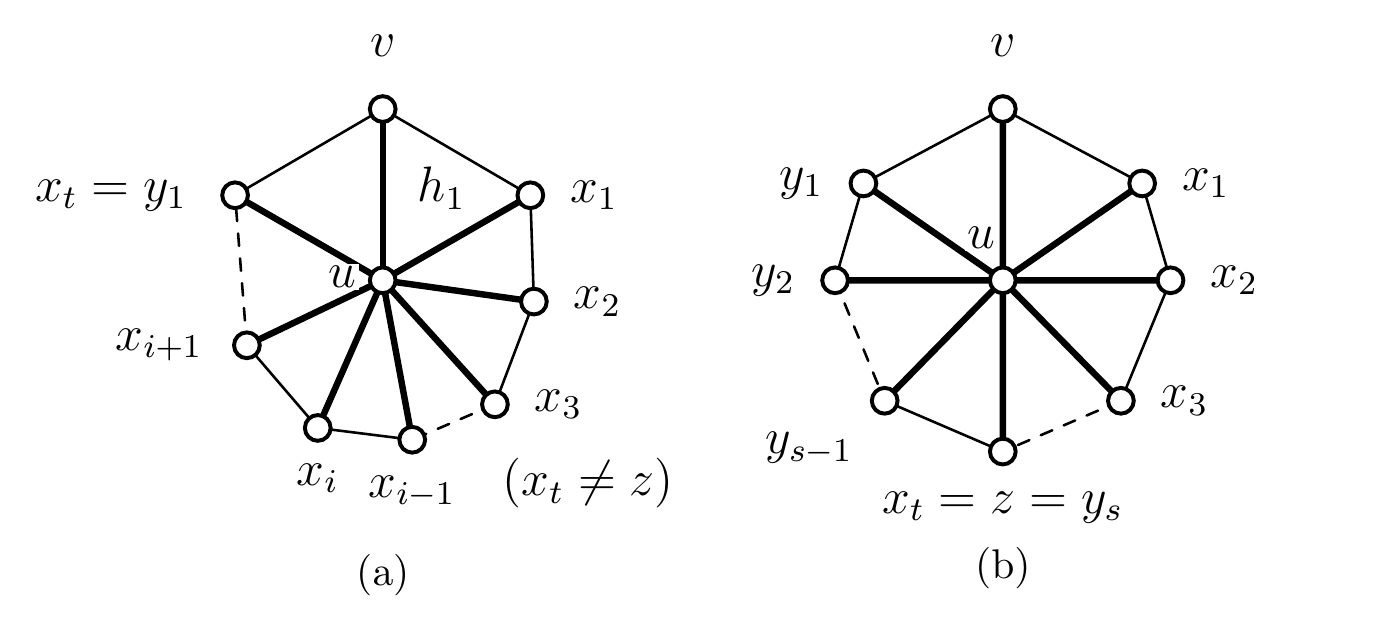}
		\caption{Illustration for the proof of
			Lemma~\ref{lem:oddwheel}: the thick edges are forcing edges.}
		\label{fig:lemma12}
	\end{figure}

	It remains to consider the case in which the sequence reaches $z$, i.e., $x_t=z$. If $y_1=z$, then
	$
		vx_1x_2\cdots x_t v
	$
	is a  rim with hub $u$, and the same argument as in the previous paragraph shows that $G$ is an odd wheel. Thus assume that $y_1\neq z$. Repeating the preceding construction starting from $y_1$, we obtain a maximal sequence
	$
		y_1,y_2,\ldots,y_s
	$
	of 3-degree vertices adjacent to $u$. If this sequence stops at $v$, then we again obtain a wheel with hub $u$. Otherwise it reaches $z$, and
	$
		vx_1x_2\cdots x_t y_{s-1}\cdots y_2y_1v
	$
	is a  rim with hub $u$, where $y_s=z=x_t$, see Figure~\ref{fig:lemma12}(b). As above, all rim vertices have degree three, so this wheel is the whole graph. Since $G$ has an even number of vertices, the rim has an odd number of vertices. Therefore $G$ is an odd wheel.
	\end{proof}


\begin{Not}\label{not:solid-multiple}
	Let $G$ be a solid brick with the property $\mathcal P$. Then the
	underlying graph of $G$ is an odd wheel. Moreover, the following
	statements hold:
	\begin{enumerate}[(i)]
	\item If $|V(G)|=4$, then the edges of multiplicity at least two
		are contained in either a cycle with length four  or the set of edges
		incident with one vertex.
		\item If $|V(G)|\geq6$, then every rim edge has multiplicity one.
		The spokes may have arbitrary positive multiplicities.
	\end{enumerate}
\end{Not}

\section{Splicing an Odd Wheel with a Brick}\label{sec:splicing}

\begin{Lem}\label{lem:rim vertex}
	Let $G_1$ be an odd wheel up to multiple edges with
	$|V(G_1)|\geq6$, let $x_1$ be a rim vertex of $G_1$, and let
	$G_2$ be a brick with $x_2\in V(G_2)$. Suppose that the graph
	$G\in G_1(x_1)\odot G_2(x_2)$ is a brick with the property $\mathcal P$.
	Then the following statements hold:
	\begin{enumerate}[(i)]
		\item
			$|N_{G_2}(x_2)|\geq4$ and $|V(G_2)|\geq6$;
		\item
			the underlying  graph of $G_2$ is neither an odd
			wheel nor isomorphic to any of
			$\{\overline{C_6},\overline{C_6}^{+},H_8\}$.
	\end{enumerate}
\end{Lem}

\begin{proof}
\noindent\textbf{(i).}
		Let $N_{G_1}(x_1)=\{c_1,x_1^-,x_1^+\}$ and let
		$C=\partial_G(V(G_1)\setminus\{x_1\})$.
	For each $\varepsilon\in\{-,+\}$, let $f_\varepsilon$ be a
	forcing edge of $G$ incident with $x_1^\varepsilon$.
	If $f_\varepsilon\in C$, then $x_1x_1^\varepsilon$ would be
	forcing in $G_1$ by Proposition~\ref{pro:M}(i), contrary to the
	fact that it is a rim edge. Hence
	$f_\varepsilon\in E(G_1-x_1)$. Again by
	Proposition~\ref{pro:M}(i), $f_\varepsilon$ is forcing in $G_1$.
	Since every rim edge of $G_1$ is nonforcing, we have
	$f_\varepsilon=c_1x_1^\varepsilon$.
	The unique perfect matching of $G_1$ containing $f_-$ uses
	$x_1x_1^+$, whereas the unique perfect matching containing $f_+$
	uses $x_1x_1^-$. Hence both $x_1x_1^+$ and $x_1x_1^-$ have
	multiplicity one. Therefore, for each $\varepsilon\in\{-,+\}$,
	the set $\partial_G(x_1^\varepsilon)\cap C$ consists of a single
	edge, say $e_\varepsilon$. By Proposition~\ref{pro:M}(i), the
	edge of $G_2$ corresponding to $e_\varepsilon$ is forcing in
	$G_2$.
	Since $G$ is a brick and $C$ is a nontrivial cut, $C$ is not a
	tight cut. Let $M$ be a perfect matching of $G$ such that
	$|M\cap C|\geq3$. {Since $\{c_1,x_1^-,x_1^+\}$ covers $C$},
	$|M\cap C|=3$.
		Write
		$M\cap C=\{c_1y_2,x_1^-x_2^-,x_1^+x_2^+\}$, where
		$\{y_2,x_2^-,x_2^+\}\subseteq
		V(G_2)\setminus\{x_2\}$, and let $e_y$ be a forcing edge of
		$G$ incident with $y_2$.
	By the preceding argument, $c_1$ is the unique
	neighbor of $y_2$ in $V(G_1)\setminus\{x_1\}$.
	However, $c_1y_2$ is not a forcing edge of $G$ by
	Proposition~\ref{pro:M} (ii). So
	$e_y\in E(G_2-x_2)$. Let $M_y$ be the perfect matching of $G$
	containing $e_y$. By Proposition~\ref{pro:M} (i),
	$M_y\cap C$ contains exactly one edge, which corresponds to a
	forcing edge of $G_1$. So $c_1$ is one end vertex of the unique
	edge in $M_y\cap C$. Let
	$M_y\cap C=\{c_1z_2\}$, where
	$z_2\in V(G_2)\setminus\{x_2,y_2\}$.
	Since $e_y$ is a forcing edge of $G$,
	$|\partial_G(z_2)\cap C|=1$ and thus
	$z_2\notin\{x_2^-,x_2^+\}$.
	From the above discussion, $x_2^-$, $x_2^+$, $y_2$, and $z_2$
	are four distinct neighbors of $x_2$ and hence
		$|N_{G_2}(x_2)|\geq4$. Since $G_2$ is a brick, it has an even
	number of vertices and hence $|V(G_2)|\geq6$.
	
\noindent\textbf{(ii).}  By {(i)},
	$|N_{G_2}(x_2)|\geq4$ and thus $\widetilde{G_2}$ is not
	isomorphic to any of $\{K_4,\overline{C_6},H_8\}$.
	If $\widetilde{G_2}$ is isomorphic to an odd wheel, then, by
	{(i)}, $x_2$ is its hub. Moreover,
	$e_y\in E(G_2-x_2)$ is forcing in $G_2$ by
	Proposition~\ref{pro:M}(i). However, every edge of $G_2-x_2$
	is a rim edge and hence is nonforcing in $G_2$, a contradiction.
	
	\begin{figure}[htbp]
	\centering
	\includegraphics[width=\textwidth]{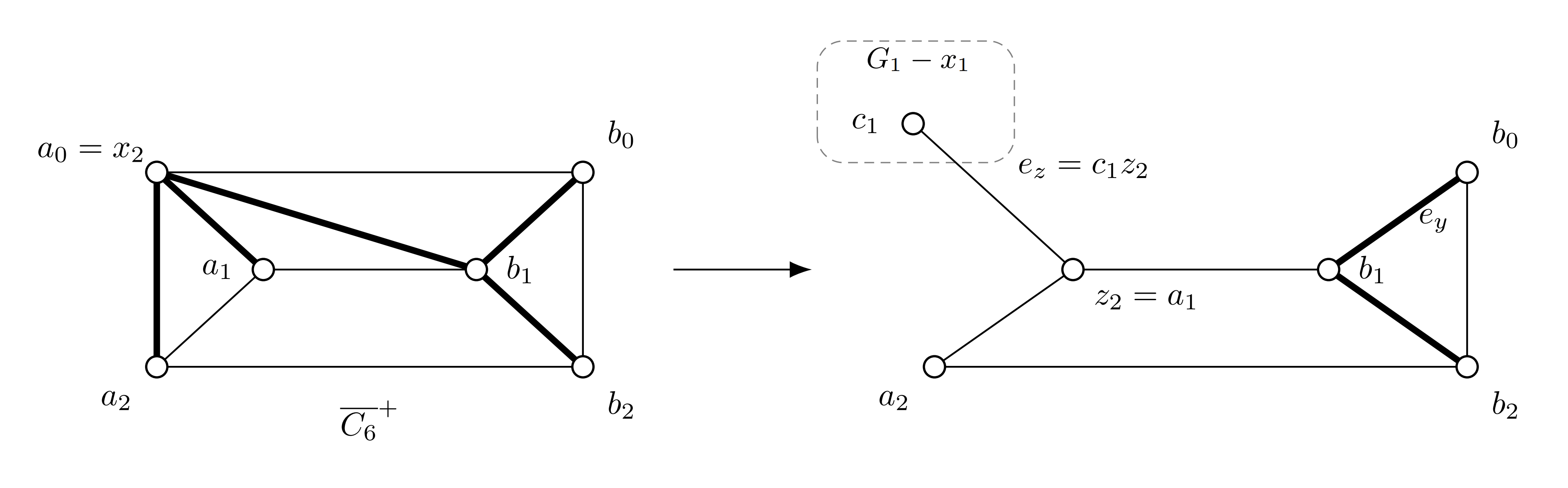}
	\caption{Illustration for the proof of
		Lemma~\ref{lem:rim vertex}(ii): the thick edges are forcing edges.}
	\label{fig:C6plus-rim}
\end{figure}

	If $\widetilde{G_2}\cong\overline{C_6}^{+}$, then
$\widetilde{G_2}$ has two $4$-degree vertices and four
$3$-degree vertices. By~(i),
	$|N_{G_2}(x_2)|\geq4$. Thus, using the vertex labels of $\overline{C_6}^{+}$  shown in
	Figure~\ref{fig:T6}, we may assume that $x_2=a_0$.
		As shown in Figure~\ref{fig:C6plus-rim}, $b_0b_1$ is the
		unique forcing edge of $G$ incident with $b_0$. So we can
		assume that $e_y=b_0b_1$ and thus $z_2=a_1$. Recall that
		$\partial_G(z_2)\cap C=\{c_1z_2\}$. Let $e_z$ be a forcing
		edge of $G$ incident with $z_2$. Then $e_z$ corresponds to a
		forcing edge of $G_2$, and thus $e_z\in C$.
		However,
		$\{e_z,x_1^-x_2^-,x_1^+x_2^+\}$ belongs to a perfect
		matching of $G$, a contradiction to
		Proposition~\ref{pro:M} (ii).
		So $\widetilde{G_2}$ is neither an odd wheel nor any of
		$\{\overline{C_6},\overline{C_6}^{+},H_8\}$.
\end{proof}

\begin{Lem}\label{lem:hub}
	Let $G_1$ be an odd wheel up to multiple edges with
	$|V(G_1)|\geq6$, let $c_1$ be its hub, and let $G_2$ be a brick with $x_2\in V(G_2)$. Suppose that
	$G\in G_1(c_1)\odot G_2(x_2)$ is a brick with the property $\mathcal P$.
	Then the following statements hold:
	\begin{enumerate}[(1)]
		\item $|N_{G_2}(x_2)|\geq4$ and $|V(G_2)|\geq6$;
		\item the underlying  graph of $G_2$ is neither an odd
		wheel nor isomorphic to any of
		$\{\overline{C_6},\overline{C_6}^{+},H_8\}$.
	\end{enumerate}
\end{Lem}

\begin{proof}
\noindent\textbf{(1).}
	Let $R_1$ be the rim of $G_1$, let $C$ be the splicing cut, and
	let $C_f$ be the set of forcing edges of $G$ contained in $C$.
	Every rim edge of $G_1$ is nonforcing in $G_1$ and hence, by
	Proposition~\ref{pro:M}(i), nonforcing in $G$. Since the hub
	$c_1$ is deleted in the splicing
	and $G$ has the property $\mathcal P$,
	$C_f$ {saturates} $V(R_1)$.	Set $G'=G-(C\setminus C_f)$.
	
	{\begin{Cla}\label{cla:Gprime-not-brick}
			The graph $G'$ is not a brick.
	\end{Cla}}
	
	\begin{proof}
		Otherwise, $C_f$ would be a nontrivial, nontight cut of $G'$.
		Since both shores of $C_f$ contain an odd number of vertices, some perfect matching
		of $G'$ would {share at least three edges with $C_f$}.
		This matching is also a perfect matching of $G$
		{sharing at least three edges with $C$}
		and containing forcing edges, contrary to
		Proposition~\ref{pro:M}(ii).
	\end{proof}
	
	Suppose, to the contrary, that $|N_{G_2}(x_2)|\leq3$. Since $G_2$
	is 3-connected, write
	$N_{G_2}(x_2)=\{q_1,q_2,q_3\}$. Since $G$ is a brick and $C$ is
	a nontrivial cut, $C$ is not tight. Hence $G$ has a perfect
	matching $M$ with $|M\cap C|\geq3$. {Since $\{q_1,q_2,q_3\}$ covers $C$}, $|M\cap C|=3$.
	Write $M\cap C=\{p_1q_1,p_2q_2,p_3q_3\}$, where
	$p_1,p_2,p_3\in V(R_1)$. By Proposition~\ref{pro:M}(ii), every
	edge of $M\cap C$ is nonforcing.
		\begin{Cla}\label{cla:Cf-covers-neighbors}
			The set $C_f$ {saturates} $N_{G_2}(x_2)$.
		\end{Cla}
	\begin{proof}
		Suppose, by
		relabeling, that $q_1$ is incident with no edge of $C_f$.
	Since $G$ has the property $\mathcal P$, $q_1$ is incident with a forcing edge
		$e_1\in E(G_2-x_2)$. By Proposition~\ref{pro:M}(i), $e_1$ is
		forcing in $G_2$. The unique perfect matching of $G_2$
		containing $e_1$ uses an edge joining $x_2$ to either $q_2$ or
		$q_3$,
		say that this end is $q_2$.
		By {Proposition~\ref{pro:splicing-degree}}, {we can deduce that $\partial_G(q_2)\cap C=\{p_2q_2\}$.
			Since $p_2q_2$} is nonforcing, $q_2$ is incident with a forcing edge
		$e_2\in E(G_2-x_2)$. Applying the same argument to $e_2$, we
		obtain either $|\partial_G(q_1)\cap C|=1$ or
		$|\partial_G(q_3)\cap C|=1$, {say the former.}
		It follows that every edge of $C_f$ is incident with $q_3$.
		Since $C_f$ {saturates} $V(R_1)$, every vertex of $R_1$ is joined
		to $q_3$ by an edge of $C_f$. Let $P$ be the unique
		$p_1$--$p_2$ path in $R_1$ containing an odd number of vertices, and let $p$ be the
		neighbor of $p_1$ on $P$. Then
		$R_1-\{p_1,p_2,p\}$ has a perfect matching. Hence
		the matching $\{p_1q_1,p_2q_2,pq_3\}$ is contained
			in a perfect matching of $G$
		{sharing three edges with $C$} and containing the
		forcing edge $pq_3$, contrary to
		Proposition~\ref{pro:M}(ii). Thus $C_f$
		{saturates} $N_{G_2}(x_2)$.
	\end{proof}
	
Since $C_f$ saturates both $V(R_1)$ and
$N_{G_2}(x_2)$, the graph $G'$ is a splicing of two graphs whose
underlying graphs are the same as those of $G_1$ and $G_2$,
respectively. Hence
	they are bricks by Proposition~\ref{pro:underlying-simple}.
	Since $G'$ is not a brick, Lemma~\ref{thm:splicing-brick} yields
	vertices $y_1\in V(R_1)$ and
	{$y_2\in \{q_1,q_2,q_3\}$ that cover $C_f$.}
	We claim that $y_1y_2\in M$. Otherwise, let $z_1y_2$ be the
	edge of $M\cap C$ incident with $y_2$.  {Since $z_1\ne y_1$, $C_f$ {saturates} $z_1$, and
		$\{y_1,y_2\}$ covers $C_f$, we have $z_1y_2\in C_f$,}
	contrary
	to Proposition~\ref{pro:M}(ii).
	Relabel the edges of $M\cap C$ so that $y_1=p_1$ and
	$y_2=q_1$. The edge $p_1q_1$ is nonforcing, and the two rim edges incident
	with $p_1$ are also nonforcing. Since $G$ has the property $\mathcal P$,
	$p_1$ is incident with another edge $f\in C_f$. If $f$ also
	joins $p_1$ and $q_1$, replacing $p_1q_1$ by $f$ gives an
	immediate contradiction. Otherwise, by relabeling
	$p_2q_2$ and $p_3q_3$, assume that $f$ joins $p_1$ and $q_2$.
	Since $C_f$ {saturates} $p_2$ and
	$\{p_1,q_1\}$ covers $C_f$, there is an edge $g\in C_f$ joining
	$p_2$ and $q_1$. Replacing $M\cap C$ by
	$\{f,g,p_3q_3\}$ gives a perfect matching
	{sharing three edges with $C$} and containing a
	forcing edge, again contrary to Proposition~\ref{pro:M}(ii).
	Therefore, $|N_{G_2}(x_2)|\geq4$. Since $G_2$ is a brick, it
	has an even number of vertices, and hence $|V(G_2)|\geq6$.
	
\noindent\textbf{(2).}
	Suppose that $\widetilde{G_2}$ is an odd wheel, with rim $R_2$.
	Since $|N_{G_2}(x_2)|\geq4$, the vertex $x_2$ is its hub.
	Moreover, the rim edges are nonforcing. Since $G$ has the property $\mathcal P$,
	$C_f$ {saturates}
	$V(R_2)=N_{G_2}(x_2)$.
	Since $C_f$ {saturates} both $V(R_1)$ and $V(R_2)$,
	the two graphs
	used to obtain $G'$ by splicing have underlying graphs
	$\widetilde{G_1}$ and $\widetilde{G_2}$, respectively, and hence
	are bricks by Proposition~\ref{pro:underlying-simple}. By
	Claim~\ref{cla:Gprime-not-brick}, $G'$ is not a brick.
	Lemma~\ref{thm:splicing-brick} yields vertices
	$y_1\in V(R_1)$ and $y_2\in V(R_2)$ such that every edge of
	$C_f$ is incident with $y_1$ or $y_2$.
	Since $G$ is a brick, Lemma~\ref{thm:splicing-brick} implies
	that $\{y_1,y_2\}$ does not cover $C$. Hence there is an edge
	$uv\in C\setminus C_f$, where
	$u\in V(R_1)\setminus\{y_1\}$ and
	$v\in V(R_2)\setminus\{y_2\}$.
There exist
	$z_1\in N_{G_1}(y_1)\cap V(R_1)$ and
	$z_2\in N_{G_2}(y_2)\cap V(R_2)$ such that
	$R_1-\{y_1,u,z_1\}$ has a perfect matching $M_1$ and
	$R_2-\{y_2,v,z_2\}$ has a perfect matching $M_2$. Since $C_f$ saturates both $V(R_1)$ and $V(R_2)$ and
	$\{y_1,y_2\}$ covers $C_f$, the edges
	$z_1y_2$ and $y_1z_2$ belong to $C_f$. Therefore,
	$M_1\cup M_2\cup\{uv,z_1y_2,y_1z_2\}$ is a perfect matching
	of $G$
	sharing three edges with $C$ and containing
	the forcing edges $z_1y_2$ and $y_1z_2$,
	contrary to Proposition~\ref{pro:M}(ii). Thus
	$\widetilde{G_2}$ is not an odd wheel.
	
	Since every vertex of $\overline{C_6}$ and $H_8$ has degree
	three, $|N_{G_2}(x_2)|\geq4$ implies that
	$\widetilde{G_2}$ is isomorphic to neither
	$\overline{C_6}$ nor $H_8$.
	Finally, suppose that
	$\widetilde{G_2}\cong\overline{C_6}^{+}$. Using the vertex labels
	for $\overline{C_6}^{+}$ shown in
	Figure~\ref{fig:T6}, the only 4-degree vertices are
	$a_0$ and $b_1$, assume that $x_2=a_0$.
		As in the proof of Lemma~\ref{lem:rim vertex}(ii),
		see Figure~\ref{fig:C6plus-rim},
		write
		$\partial_G(a_1)\cap C=\{e\}$. Since all other edges incident
		with $a_1$ are nonforcing, $e$ is forcing. Choose a perfect
		matching $M'$ of $G$ with $|M'\cap C|\geq3$. If $e\notin M'$, then $M'\cap C$ saturates
		$\{a_2,b_0,b_1\}$. The graph
			$(G_2-x_2)-\{a_2,b_0,b_1\}$ consists of the two isolated
			vertices $a_1$ and $b_2$, contradicting the fact that $M'$ is
		a perfect matching of $G$. Thus $e\in M'$, contrary to
		Proposition~\ref{pro:M}(ii).
	So $\widetilde{G_2}$ is neither an odd wheel nor any of
	$\{\overline{C_6}, \overline{C_6}^{+},H_8\}$.
	%
	%
\end{proof}

\begin{Lem}\label{lem:k4-splicing}
	Let $K$ and $B$ be graphs such that $\widetilde K\cong K_4$ and
	$\widetilde B$ is isomorphic to one of
	$\{K_4,\overline{C_6},\overline{C_6}^{+},H_8\}$. Let $G$ be a
	simple brick obtained by splicing $K$ and $B$. If $G$ has
		property $\mathcal P$, then $G$ is isomorphic to one of
		$\{\overline{C_6},\overline{C_6}^{+}, H_8\}$.
\end{Lem}

\begin{proof}
	Let $u$ and $v$ be the splicing vertices of $K$ and $B$,
	respectively, let $C$ be the splicing cut. Write
	$V(K)\setminus\{u\}=\{p_1,p_2,p_3\}$. By
	Lemma~\ref{thm:splicing-brick}, the cut $C$ contains a matching
	$M$ with $|M|=3$; otherwise, two vertices, one from each
	shore, would cover $C$.

	\medskip
	\noindent\textbf{Case 1.}
	$\widetilde{B}\cong K_4$.
	
	Write $V(B)\setminus\{v\}=\{b_1,b_2,b_3\}$. Every vertex of these two triangles
	is incident with $C$. Write
	$M=\{p_1b_1,p_2b_2,p_3b_3\}$, by relabeling if necessary.
	Then $G-(C\setminus M)\cong \overline{C_6}$. Hence, if $C=M$, then
	$G\cong \overline{C_6}$, whereas if $|C\setminus M|=1$, then
	$G\cong \overline{C_6}^{+}$; see Figure~\ref{fig:T6}.
	Suppose that $|C\setminus M|\geq2$. If some vertex of either
	triangle is incident with no edge of $C\setminus M$, then we may
	interchange the two triangles and relabel their vertices so that
	$p_3$ is incident with no edge of $C\setminus M$, whereas both
	$p_1$ and $p_2$ are incident with such an edge. Indeed, when
	$|C\setminus M|=2$, the two edges have distinct ends in at least
	one triangle because $G$ is simple; when $|C\setminus M|\geq3$,
	the assertion follows since each vertex is incident with at most
	two edges of $C\setminus M$. Since $M$ is a perfect matching of
	$G$ {sharing three edges with $C$},
	Proposition~\ref{pro:M}(ii) implies
	that $p_3b_3$ is nonforcing. Moreover, both $p_1p_3$ and $p_2p_3$ belong to at least two
	perfect matchings of $G$ and hence are nonforcing. Thus $p_3$ is not incident with any forcing edge, a
	contradiction.
	It remains to consider the case in which every vertex of $G$ is
	incident with an edge of $C\setminus M$. Hence every vertex is
	incident with at least two edges of $C$, and every edge in
	$E(G)\setminus C$ is nonforcing. Let $e\in C$. After deleting
	the two ends of $e$, each remaining vertex is still incident
	with an edge of $C$. Thus $C\setminus\{e\}$ contains a matching
	with two edges that saturates the four remaining vertices. Together
	with $e$, this matching forms a perfect matching of $G$.
	By Proposition~\ref{pro:M}(ii), the edge $e$ is
	nonforcing. Thus $G$ contains no forcing edge, contrary to
	the assumption that $G$ has the property $\mathcal P$.
	Therefore, if the underlying  graph of $B$ is isomorphic to
	$K_4$, then $G\cong \overline{C_6}$ or
	$G\cong \overline{C_6}^{+}$.
	
	\medskip
	\noindent\textbf{Case 2.}
	$\widetilde{B}\cong \overline{C_6}$.

	We use the vertex labels for $\overline{C_6}$ shown in
	Figure~\ref{fig:T6}. Since $\overline{C_6}$ is vertex-transitive,
	we may assume that $v=x$. 
	 The vertices of
	$B-v$  incident with some edge in $C$ are precisely $y,z,x'$.    
	The matching $M$ {saturates
	$\{p_1,p_2,p_3,y,z,x'\}$}.
	We may write
	$M=\{p_1y,p_2z,p_3x'\}$ after relabeling
	$p_1,p_2,p_3$ if necessary. Thus $M\cup\{y'z'\}$ is a perfect
	matching of $G$ {sharing three edges with $C$}. By
	Proposition~\ref{pro:M}(i), the edge $y'z'$ is nonforcing.	
	\begin{figure}[htbp]
		\centering
		\includegraphics[width=\textwidth]{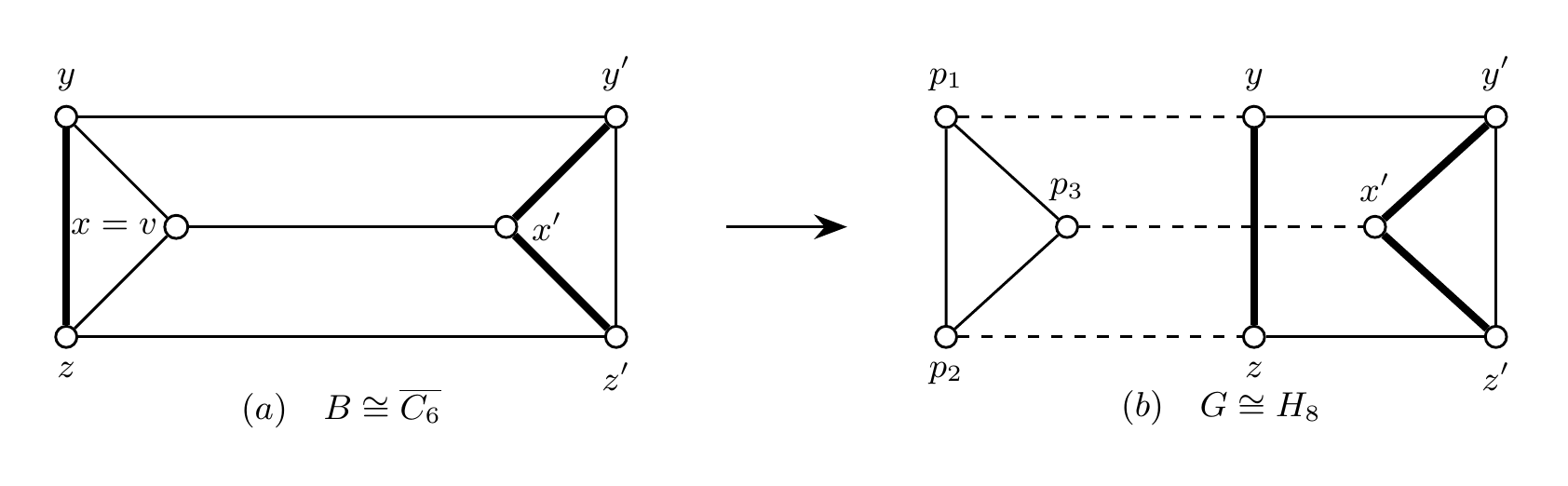}
		\caption{Illustration for Case~2 in the proof of
			Lemma~\ref{lem:k4-splicing}: the thick edges are forcing edges,
			and the dashed edges belong to the splicing cut $C$.}
		\label{fig:k4-splicing-case2}
	\end{figure}
	The edges $yy'$ and $zz'$ belong to at least two perfect
	matchings of $\widetilde B~ (\cong \overline{C_6})$ and hence are
	nonforcing in $B$. By Proposition~\ref{pro:M}(i), they are also
	nonforcing in $G$. Since $G$ has the property $\mathcal P$, both $y'$ and $z'$ are incident with forcing edges of $G$.
	By the preceding discussion,
	$x'y'$ and $x'z'$ are forcing in $G$ and, by
	Proposition~\ref{pro:M}(i), forcing in $B$.	
	The unique perfect matchings of $B$ containing $x'y'$ and $x'z'$
	use the edges $xy$ and $xz$, respectively. By
	{Proposition~\ref{pro:splicing-degree}},
	$|\partial_G(y)\cap C|=|\partial_G(z)\cap C|=1$. Hence $p_1y$
	and $p_2z$ are the only edges of $C$ incident with $y$ and $z$,
	respectively.
	Since $M\cup\{y'z'\}$ is a perfect matching of $G$
	{sharing three edges with $C$},
	Proposition~\ref{pro:M}(ii) implies that $p_1y$ is
	nonforcing. The edge $yy'$ is also nonforcing.
	Since $G$ has the property $\mathcal P$, $yz$ is forcing in $G$, {and thus is forcing in $B$ by
		Proposition~\ref{pro:M}(i). The unique perfect matching of $B$ containing $yz$ uses the edge $xx'$, }and hence
	{Proposition~\ref{pro:splicing-degree}} gives
	$|\partial_G(x')\cap C|=1$. Consequently,
	$C=\{p_1y,p_2z,p_3x'\}$.
	Therefore, $G\cong H_8$; see
	Figure~\ref{fig:k4-splicing-case2}.

	\medskip
	\noindent\textbf{Case 3.}
	$\widetilde B\cong\overline{C_6}^{+}$ or $H_8$.
	
	Using the vertex labels of the graphs $\overline{C_6}^{+}$ and $H_8$ in Figure~\ref{fig:T6}, a direct check
	shows that the forcing edges of $\overline{C_6}^{+}$ are
	$a_0a_1,a_0a_2,a_0b_1,b_0b_1,b_1b_2$, whereas those of $H_8$
	are $a_0a_1,a_0a_2,b_0b_1,b_0b_2,s_1s_2$. Since $G$ is simple,
	every multiple edge of $B$ is incident with $v$. Note that every
	edge of $B-v$ that is nonforcing in $\widetilde B$ is also
	nonforcing in $B$.
\begin{Cla}\label{cla:case3-forcing-at-v}
	The splicing vertex $v$ is incident with at least two forcing
	edges of $\widetilde B$.
\end{Cla}
\begin{proof}
Suppose, to the contrary, that $v$ is incident with exactly one
forcing edge of $\widetilde B$. Then $d_{\widetilde B}(v)=3$.
Write $N_{\widetilde B}(v)=\{q_1,q_2,q_3\}$, where $vq_1$ is
the unique forcing edge of $\widetilde B$ incident with $v$.
We may write $M=\{p_1q_1,p_2q_2,p_3q_3\}$ after relabeling
$p_1,p_2,p_3$ if necessary. In either
$\overline{C_6}^{+}$ or $H_8$, a direct check shows that
$B-\{v,q_1,q_2,q_3\}$ has a perfect matching. Hence
$M$ is contained in a perfect matching of $G$ sharing three edges with
$C$, and Proposition~\ref{pro:M}(ii) implies that every edge
of $M$ is nonforcing.
We claim that no forcing edge of $C$ is incident with $p_1$.
	Indeed, suppose that $p_1q\in C$ is forcing in $G$. By
	Proposition~\ref{pro:M}(ii), $vq$ is forcing in $B$ and hence
	in $\widetilde B$. Since $vq_1$ is the unique forcing edge of
	$\widetilde B$ incident with $v$, we have $q=q_1$. Thus
	$p_1q=p_1q_1$, which is nonforcing, a contradiction.
Since $G$ has the property $\mathcal P$, the vertex $p_1$ is therefore incident
with a forcing edge $e_1\in E(K-u)$, say $e_1=p_1p_2$. The
unique perfect matching of $K$ containing $e_1$ uses the edge
$up_3$.
	Proposition~\ref{pro:splicing-degree} gives
	$\partial_G(p_3)\cap C=\{p_3q_3\}$. Thus $up_3$ corresponds
	to $vq_3$, and Proposition~\ref{pro:M}(i) implies that $vq_3$
	is forcing in $B$, and hence in $\widetilde B$.
This contradicts the choice of $vq_1$. Thus $v$ is incident with
at least two forcing edges of $\widetilde B$.
\end{proof}

	By Claim~\ref{cla:case3-forcing-at-v} and the above lists of
	forcing edges, up to an automorphism of $\widetilde B$, it remains
	only to consider
	$(\widetilde B,v)=(\overline{C_6}^{+},a_0)$ and
	$(\widetilde B,v)=(H_8,a_0)$. First suppose that
	$(\widetilde B,v)=(\overline{C_6}^{+},a_0)$. As in the
	$\overline{C_6}^{+}$ case in the proof of
	Lemma~\ref{lem:rim vertex}(ii),
	there is a forcing edge $e\in C$ such that
	$\partial_G(a_1)\cap C=\{e\}$.
	After relabeling $p_1,p_2,p_3$ if necessary, write
		$e=p_1a_1$. The edge $e$ is contained in a matching
	$M_e\subseteq C$ with $|M_e|=3$.
	Otherwise, all edges of $C$ incident with
		$p_2$ or $p_3$ would have a common end $w\in V(B-v)$, so
		$\{p_1,w\}$ would cover $C$, contrary to
		Lemma~\ref{thm:splicing-brick}.
	After deleting from $B-v$ all vertices incident with an edge of
	$M_e$, the two remaining vertices are adjacent.
	Let $f$ be the edge joining them. Then
		$M_e\cup\{f\}$ is a perfect matching of $G$
	sharing three edges with $C$ and containing the forcing edge $e$,
	contrary to Proposition~\ref{pro:M}(ii). Now suppose that
	$(\widetilde B,v)=(H_8,a_0)$. 
		In this case, $b_0b_2$ is the only edge of $G$ incident with
		$b_2$ that can be forcing. Hence the property $\mathcal P$
		implies that $b_0b_2$ is forcing in $G$. Its unique perfect
		matching in $B$ contains $a_0a_1$, so
		Proposition~\ref{pro:splicing-degree} gives
		$\partial_G(a_1)\cap C=\{e'\}$. Since the other edges incident
		with $a_1$ are nonforcing, $e'$ is forcing in $G$.
	 After relabeling $p_1,p_2,p_3$ if necessary, write
		$e'=p_1a_1$. The same argument used for $e$ shows
		that $e'$ is contained in a matching
		$M_{e'}\subseteq C$ with $|M_{e'}|=3$.
	Since the only vertices of $B-v$ incident with $C$
		are $a_1,a_2,b_0$, the graph obtained from $B-v$ by deleting
	the vertices incident with the edges of $M_{e'}$ has the
	perfect matching $\{b_1s_1,b_2s_2\}$. Thus
	$M_{e'}\cup\{b_1s_1,b_2s_2\}$ is a perfect matching of $G$
	containing the forcing edge $e'$ and sharing three edges with
	$C$, contrary to Proposition~\ref{pro:M}(ii).
\end{proof}

	 \section{Proof of Theorem~\ref{thm:intro-main}}\label{sec:main}

%
%
%
%
%

	\begin{Lem}\label{cla:reduction}
			If $G$ is a nonsolid brick with the property $\mathcal P$, then
			$G$ has a robust cut $C$ such that one of the $C$-contractions
			is an odd wheel up to multiple edges, and the
			other
			is a brick with the property
			$\mathcal P$.
	\end{Lem}

	\begin{proof}
Applying Theorem~\ref{lem:robustcut}, there exist  $X\subset V(G)$ and two vertex-disjoint  subsets
			$X'\subseteq X$ and $X''\subseteq\overline X$
		such that $G/\overline{X'}$ is a solid brick,
		$G/\overline{X''}$ is a brick, and the graph $H$ obtained from
		$G$ by contracting $X'$ and $X''$ to single vertices $x'$ and
		$x''$, respectively, is bipartite and matching covered.
		Let $W=G/(\overline{X'}\rightarrow{w})$ and
		$G'=G/(\overline{X''}\rightarrow{w'})$. {
		We claim that $V(H)=\{x',x''\}$. Otherwise,}  let
		$z\in V(H)\setminus\{x',x''\}$. Since $G$ has the property
		$\mathcal P$, there is a forcing edge $e$ of $G$ incident with
		$z$. Let $e_H$ be the edge of $H$ corresponding to $e$. Since
		$e_H$ is incident with $z\notin\{x',x''\}$, it does not join
		$x'$ and $x''$. {If $e_H$ is not forcing in $H$,} then two
		distinct perfect matchings of $H$ containing $e_H$ give
		two distinct perfect matchings of $G$ containing $e$, a
		contradiction. Hence $e_H$ is forcing in $H$, contrary to
		Lemma~\ref{lem:bipartite}. Therefore, $V(H)=\{x',x''\}$. Hence $X'=X$ and
		$X''=\overline X$. Consequently, with
		$C=\partial_G(X)$, the graphs $W$ and $G'$ are the two
		$C$-contractions of $G$, and
		$G\cong W(w)\odot G'(w')$.

			By Proposition~\ref{pro:M}, all vertices of $W$, except
			possibly the contraction vertex $w$, are incident with forcing
			edges of $W$. By Lemma~\ref{lem:oddwheel}, $W$ is an odd wheel
			up to multiple edges. Similarly,
			Proposition~\ref{pro:M} shows that every vertex of $G'-w'$ is
			incident with a forcing edge of $G'$.
			Choose $y\in V(W)\setminus\{w\}$ and a forcing edge $f$ of
			$G$ incident with $y$. Whether $f\in E(W-w)$ or $f\in \partial_G(X)$,
			Proposition~\ref{pro:M}(i) or (ii), respectively, yields a
			forcing edge of $G'$ incident with $w'$. Hence $G'$ has the
			property $\mathcal P$.
	\end{proof}

  \begin{proof}[\normalfont\bfseries
  	{Proof of Theorem~\ref{thm:intro-main}.}]
	The sufficiency follows by direct verification. We prove the necessity
	by strong induction on the number of robust cuts of $G$.
	If $G$ has no robust cut, then Lemma~\ref{cla:reduction} implies that
	$G$ is solid, and hence Lemma~\ref{lem:oddwheel} implies that $G$ is
	an odd wheel. Now suppose that $G$ has at least one robust cut. Then
	$G$ is nonsolid.
  		Apply Lemma~\ref{cla:reduction} and use the notation introduced
  		in its proof.
  	Then $G\cong W(w)\odot G'(w')$,
		where $W$ is an odd wheel up to multiple edges and $G'$ is a
		brick with the property $\mathcal P$.
	By Proposition~\ref{pro:underlying-simple}, $\widetilde{G'}$ is a brick with the property $\mathcal P$.
	Every robust cut of $\widetilde{G'}$ corresponds
	to a robust cut of $G$ distinct from $C$. Hence $\widetilde{G'}$ has fewer robust cuts than $G$. Hence the induction
  	hypothesis implies that it is either an odd wheel or isomorphic to
  	one of $\{\overline{C_6},\overline{C_6}^{+},H_8\}$.
  	If $\widetilde{W}$ is an odd wheel with at least six
  	vertices, then
  	Lemma~\ref{lem:rim vertex}(ii) or Lemma~\ref{lem:hub}(2),
  	according as $w$ is a rim vertex or the hub, gives a contradiction.
  	Hence $\widetilde{W}$ is isomorphic to $K_4$.
  	The same argument, with $W$ and $G'$ interchanged, shows that $\widetilde{G'}$ cannot be an odd wheel with at least six
  	vertices. Therefore,
  	it is isomorphic to one of
  	$\{K_4,\overline{C_6},\overline{C_6}^{+},H_8\}$.
  	Lemma~\ref{lem:k4-splicing} now implies that $G$ is isomorphic to
  	one of $\{\overline{C_6},\overline{C_6}^{+},H_8\}$.
  \end{proof}

    \begin{Not}\label{not:exceptional-multiple}
    	Let G be a brick with the property $\mathcal{P}$ and $\widetilde{G}\in \{\overline{C_6}, C_6^+, H_8\}$. Using the vertex labels
	for these three graphs shown in
	Figure~\ref{fig:T6}, the following statements hold:
    	\begin{enumerate}[(i)]
    		\item If $\widetilde G\cong\overline{C_6}$, then, after
    		relabeling the vertices if necessary, one of the following holds:
    		either every edge other than $xy$ and $x'z'$ has multiplicity
    		one, or every edge other than $yz$, $y'z'$, and $xx'$ has
    		multiplicity one.
    		\item If $\widetilde G\cong\overline{C_6}^{+}$, then every
    		edge outside $\{a_0a_2,b_1b_2,a_0b_1\}$ has multiplicity one.
    		\item If $\widetilde G\cong H_8$, then every edge other than
    		$s_1s_2$ has multiplicity one.
    	\end{enumerate}
    \end{Not}

    \section*{Acknowledgements}

   The authors gratefully acknowledge financial support from NSFC
   (No. 12271235), the Startup Research Fund
   for Introduced Talents of Sichuan Normal University
   (No. kyqd20260308), the Natural Science Foundation of Fujian Province
   (No. 2026J002034) and Sichuan Science and Technology Program
   (No. 2026NSFSC2032).


\begin{thebibliography}{99}\setlength{\itemsep}{0mm}\linespread{1.2}\selectfont

  \bibitem{Adams2004}
  P. Adams, M. Mahdian, E.S. Mahmoodian, On the forced matching numbers of bipartite graphs, Discrete Math. 281 (2004) 1-12.
  \bibitem{CLM2004}
  M.H. de Carvalho, C.L. Lucchesi, U.S.R. Murty, The perfect matching polytope and solid bricks, J. Combin. Theory Ser. B 92 (2004) 319-324.
  \bibitem{Carvalho2004}
M.H. de Carvalho, C.L. Lucchesi, U.S.R. Murty, Graphs with independent perfect matchings, J. Graph Theory 48 (2005) 19-50.
  \bibitem{Goedgebeur2026}
J. Goedgebeur, D. Mattiolo, G. Mazzuoccolo, J. Renders, I.H. Wolf, Cubic graphs with edges in exactly one perfect matching, J. Graph Theory 112 (2026) 276-289.
  \bibitem{Harary1991}
  F. Harary, D.J. Klein, T.P. \v{Z}ivkovi\'{c}, Graphical properties of polyhexes: perfect matching vector and forcing, J. Math. Chem. 6 (1991) 295-306.
  \bibitem{He2026}
  X. He, F. Lu, J. Xue, Wheel-like bricks and minimal matching covered graphs, J. Graph Theory 111 (2026) 5-16.
  \bibitem{Jiang2011}
  X. Jiang, H. Zhang, On forcing matching number of boron-nitrogen fullerene graphs, Discrete Appl. Math. 159 (2011) 1581-1586.
  \bibitem{Klein1987}
  D.J. Klein, M. Randi\'{c}, Innate degree of freedom of a graph, J. Comput. Chem. 8 (1987) 516-521.
  \bibitem{Lovasz1987}
L. Lov\'{a}sz, Matching structure and the matching lattice, J. Combin. Theory Ser. B 43 (1987) 187-222.

  \bibitem{Lovasz1986}
  L. Lov\'asz, M.D. Plummer, Matching Theory, Ann. Discrete Math. 29, Elsevier Science, Amsterdam, 1986.
  \bibitem{Lucchesi2024}
C.L. Lucchesi, U.S.R. Murty, Perfect Matchings: A Theory of Matching Covered Graphs, Springer, Cham, 2024.
\bibitem{Narayana2024}
D.V.V. Narayana, D. Mattiolo, K. Gohokar, N. Kothari,
CLM's dependence relation, solitary patterns and $r$-graphs,
arXiv:2409.00534 (2024).
  \bibitem{Ye2016}
  Y. Wu, D. Ye, C.-Q. Zhang, Uniquely forced perfect matching and unique 3-edge-coloring, Discrete Appl. Math. 215 (2016) 203-207.
  \bibitem{Zhang2025b}
  Y. Zhang, X. He, Q. Liu, H. Zhang, Forcing, anti-forcing, global forcing and complete forcing on perfect matchings of graphs: a survey, Discrete Appl. Math. 376 (2025) 318-347.
  \bibitem{Zhang2022}
  Y. Zhang, H. Zhang, Relations between global forcing number and maximum anti-forcing number of a graph, Discrete Appl. Math. 311 (2022) 85-96.

\end{thebibliography}
\end{document}